\documentclass[reqno,11pt]{amsart}
\usepackage{amsfonts}
\usepackage{xcolor}
\usepackage{latexsym}
\usepackage{url}
\usepackage{hyperref}
\usepackage{amsthm}      
\usepackage{amsmath}   
\usepackage{amssymb}
\usepackage{graphicx}  
\usepackage{multirow}  
\usepackage{indentfirst}  
\usepackage{bm} 
\usepackage{nicefrac}
 \usepackage{amsthm} 
 \usepackage{amsmath} 
 \usepackage{amsfonts}
 \usepackage{mathrsfs}
 \usepackage{enumerate}
 \usepackage{cite}  
 \usepackage{tcolorbox}
 \usepackage{color}
 \usepackage{arydshln}
 \usepackage{extarrows}
 \usepackage{tikz}
 \usepackage{amscd}

\newtheorem{theorem}{Theorem}[section]
\newtheorem{lemma}[theorem]{Lemma}
\newtheorem{proposition}[theorem]{Proposition}
\newtheorem{corollary}[theorem]{Corollary}

 \newcommand{\ho}{\mathrm{Hom}}
  \newcommand{\tr}{\mathrm{Tr}}
   \newcommand{\rank}{\mathrm{rank}}

 \author{Yang Huang}
\address{Yang Huang, School of Mathematics and Statistics, 
Central South University, 
Changsha, Hunan, 410083, P.R. China}
\email{\url{FairyHuang@csu.edu.cn}}

\author{Yongtao  Li}
\address{ Yongtao Li, College of Mathematics and Econometrics, 
Hunan University, 
Changsha, Hunan, 410082, P.R. China}
\email{\url{ytli0921@hnu.edu.cn} }

\author{Weijun Liu}
\address{Weijun Liu, School of Mathematics and Statistics, 
Central South University, 
Changsha, Hunan, 410083, P.R. China}
\email{\url{wjliu6210@126.com}}

\author{Lihua Feng}
\address{Lihua Feng, School of Mathematics and Statistics, 
Central South University, 
Changsha, Hunan, 410083, P.R. China}
\email{\url{fenglh@163.com}}

\date{\today}

\begin{document}

\title[Properties]{Some properties for morphism of representations}

\begin{abstract}
Let $\psi : G\to GL(V)$ and $\varphi :G \to GL (W)$ be representations of finite group $G$. 
A linear map $T: V\to W$ is called a morphism from $\psi$ to $\varphi$ if it
satisfys $T\psi_g= \varphi_g T$ for each $g\in G$ 
and let $\ho_G (\psi ,\varphi)$ denote the set of all morphisms.    
In this paper, we make full stufy of the subspace $\ho_G(\psi, \varphi)$. 
As byproducts, we  include the proof of the first orthogonality relation and 
Schur's orthogonality relation. 
\end{abstract}
\maketitle

\section{Introduction}
\label{sec1}

The main goal of  finite group representation theory 
is to study groups via their actions on some specific vector spaces, 
see, e.g., \cite{FH91, JL01, Se77} for details.  
Consideration of finite groups acting on finite sets leads to 
such significant results as the Sylow theorems, and also arises 
increasing concerns to the study of permutation group theory.  
By investigating group actions on vector spaces even more 
pellucid informations about a group can be obtained, 
such as Burnside's $pq$-theorem \cite{DG70}. 
This is the main subject of  representation theory. 
Meanwhile, vector spaces and linear operators are main topics 
in the study of linear algebra.   
Our study concentrating on matrix representations of groups will 
lead us naturally to the study of complex-valued functions and even 
Fourier analysis on a group.  
This in turn has extensive applications to other disciplines like graph theory, 
functional analysis and probability, just to name a few. 

We now review some basic definitions and notations of 
representation theory which will be used later. 
Let $G$ be a finite group and $V$ be a finite-dimensional vector space 
over the complex field $\mathbb{C}$. 
We define the {\it general linear group} $GL(V)$ to be the group of 
all invertible linear transformations on $V$ and 
$GL_n(\mathbb{C})$ to be the group of all invertible matrices 
over $\mathbb{C}$ of order $n$. 
We use notation $1_V$ to mean the {\it identity linear transformation} of $V$,  
and $I_n$ to denote the identity matrix of order $n$, 
or simply by $I$ if no confusion is possible.  
Let $M_{m\times n}(\mathbb{C})$ denote by the space of $m\times n$ complex matrices 
and $U_n(\mathbb{C})$ denote by the space of $n\times n$ unitary matrices. 
$A^*$ stands for the conjugate transpose of matrix $A$.
A {\it representation} of $G$ on $V$ is a group homomorphism 
$\rho : G\to GL(V)$. The dimension of $V$ is called the {\it degree} 
of $\rho$ and denoted by $\mathrm{deg}\, \rho$. 
We also say that $V$ is the {\it representation space} 
of $G$ with respect to the representation $\rho$ and 
usually write $\rho_g$ for $\rho (g)$ simply. 
If $\rho_g$ is an unitary matrix for each $g\in G$, we then say that 
$\rho$ is an unitary matrix representation. 
Two representations $\psi : G\to GL(V)$ and $\varphi : 
G\to GL(W)$ are said to be {\it equivalent} or 
{\it isomorphic}, written as $\psi \sim \varphi$,  
if there exists a linear isomorphism $T: V\to W$ such that 
$\varphi_g =T\psi_g T^{-1}$ for all $g\in G$, i.e., 
$\varphi_g T=T\psi_g$ for all $g\in G$. In pictures, 
we have that the diagram 
\[ \begin{CD}
V @>\psi_g >> V \\
@VTVV    @VVTV \\
W @>>\varphi_g>  W
\end{CD} \]
commutes for all $g\in G$.  

Let $\rho : G\to GL(V)$ be a representation. 
A subspace $W$ of $V$ is called {\it $G$-invariant under $\rho$} 
if for all $w\in W$ and $g\in G$, one has $\rho_g(w)\in W$. 
If $W$ is a $G$-invariant subspace under $\rho$, 
we can restrict $\rho$ on $W$ and obtain a new representation 
$\rho |_W: G\to GL(W)$ by setting $(\rho |_{W})_g(w)=\rho_g(w)$ 
for every  $w\in W$. 
We say $\rho |_W$ is a {\it subrepresentation} of $\rho$.  
Trivially, $\{0\}$ and $V$ are always 
$G$-invariant subspaces. 
If the only $G$-invariant subspace of $V$ are $\{0\}$ and $V$, 
we say that $\rho$ is an {\it irreducible representation}.  
The {\it character} $\chi_{\rho}: G \to \mathbb{C}$ 
of $\rho$ is defined by setting $\chi_{\rho}(g)=\tr (\rho_g)$ 
for $g\in G$, where $\tr (\rho_g)$ is the trace of the representation matrix 
of $\rho_g$ with respect to a specified basis of $V$. 
The {\it degree} of the character $\chi_{\rho}$ is defined 
to be the degree of $\rho$, which also equals $\chi_{\rho}(1)$.  
The character of an irreducible representation is called an 
{\it irreducible character}. 
It is easy to see that $\chi_{\psi}=\chi_{\varphi}$ if $\psi \sim \varphi$, 
because $\chi_{\psi}(g)=\tr (\psi_g)=\tr (T^{-1}\varphi_gT)=\tr (\varphi_g)=\chi_{\varphi}(g)$.

Let $G$ be a group and define $L(G)=\mathbb{C}^G=\{f \,|\, f:G\to \mathbb{C}\}$. 
It is well known that $L(G)$ is an inner product space with usual addition and 
scalar multiplication, and inner product defined by 
$ \langle f_1,f_2 \rangle =\frac{1}{|G|}\sum_{g\in G} f_1(g)\overline{f_2(g)}. $
Let $\psi : G\to GL(V)$ and $\varphi :G\to GL(W)$ be representations 
of finite group $G$.  
A {\it morphism} from $\psi$ to $\varphi$ is a linear map $T: V\to W$ such that 
$T\psi_g =\varphi_gT $ for all elements $g\in G$. 
Let $\ho_G(\psi ,\varphi)$ denote by the set of all morphisms from $\psi$ 
to $\varphi$ and $\ho (V,W)$ denote by the set of all 
linear maps from $V$ to $W$, see \cite{BS12} for more details.

In this paper, we first study and present some useful properties of the set 
$\ho_G (\psi ,\varphi)$, 
and then extend the Schur lemma. At the end of the paper, 
we include the proof of the first orthogonality relation and 
Schur's orthogonality relation.

\newpage 
\section{Properties of Morphism Space $\ho_G (\psi, \varphi)$} 

We first introduce some lemmas for latter use. 

\begin{lemma}[{\cite[p. 20]{BS12}}]  \label{lem21}
Every representation of a finite group $G$ is equivalent to 
a unitary matrix representation. 
\end{lemma}

\begin{lemma}[Maschke]
Every representation of a finite group is completely reducible, i.e., 
if $\rho :G\to GL(V)$ is a representation of a finite group $G$,  then 
there exists a decomposition 
$V=V_1\oplus V_2 \oplus \cdots \oplus V_t$ such that 
$V_i$ is $G$-invariant subspace and the subrepresentation $\rho|_{V_i}$ is irreducible 
for all $i=1,2,\ldots ,t$. 
\end{lemma}

The following Proposition \ref{prop24} is easy to prove by basic linear algebra techniques. 
For completeness, we include its proof below.

\begin{proposition} \label{prop24}
Let $\psi : G\to GL(V)$ and $\varphi :G\to GL(W)$ be representations.  
Then $\ho_G(\psi, \varphi )$ is a subspace of $\ho (V,W)$. 
Furthermore, if $T\in \ho_G(\psi ,\varphi)$, then 
$\mathrm{Ker}(T)$ is a $G$-invariant subspace of $V$ under $\psi$ and 
$T(V)=\mathrm{Im}(T)$ is a $G$-invariant subspace of $W$ under $\varphi$. 
\end{proposition}

\begin{proof}
Let $T_1,T_2\in \ho_G(\psi , \varphi )$ and $c_1,c_2\in \mathbb{C}$. Then for each $g\in G$, 
\[ (c_1T_1+c_2T_2)\psi_g =
c_1T_1\psi_g +c_2T_2\psi_g =
c_1\varphi_gT_1 +c_2\varphi_gT_2=\varphi_g(c_1T_1 +c_2T_2), \]
and hence $c_1T_1 +c_2T_2 \in \ho_G(\psi ,\varphi )$, as required.  

Moreover, if $T\in \ho_G(\psi ,\varphi)$, let $v \in \mathrm{Ker}(T)$ and $g\in G$. 
Then $T\psi_g v=\varphi_g T v=0$ 
since $v\in \mathrm{Ker}(T)$. 
Hence $\psi_g v\in \mathrm{Ker}(T)$, we conclude that $\mathrm{Ker}(T)$ is $G$-invariant subspace of 
$V$ under $\psi$. 
Now let $w\in \mathrm{Im}(T)$, say $w=Tv$ for some $v\in V$. 
Then $\varphi_g w=\varphi_g Tv=T\psi_g v\in \mathrm{Im}(T)$, 
it means that $\mathrm{Im}(T)$ is $G$-invariant. 
\end{proof}

\begin{proposition} \label{lem31}
If $\varphi \sim \psi$, then $\ho_G(\varphi, \rho) $ is isomorphic to $ \ho_G(\psi, \rho)$. 
\end{proposition}

\begin{proof}
Since $\varphi \sim \psi$, 
there is an isomorphism linear map $T$ such that 
$\psi_g T =T\varphi_g$ for every $g\in G$.  
Define a map $\pi : \ho_G(\psi , \rho) \to \ho_G(\varphi , \rho)$ by 
$\pi (F)= FT$ for every $F \in \ho_G(\psi , \rho) $.  
We need to prove that $\pi$ is an isomorphism. 
Firstly, the map $\pi $ is well-defined, because 
$\rho_g (FT)=(\rho_g F)T=(F\psi_g) T=F(\psi_g T)=F(T\varphi_g)=(FT)\varphi_g$, 
which means $\pi (F) =FT\in \ho_G(\varphi , \rho)$.  
As $T$ is an invertible linear map, 
it is easy to get that $\pi$ is a linear injection and surjection, i.e., 
the  map $\pi$ is an isomorphism. 
\end{proof} 

By making full use of Proposition \ref{prop24}, 
Schur proved that 

\begin{lemma}[Schur's lemma] \label{schur}
Let $\psi ,\varphi$ be irreducible representations of a finite group $G$ 
and $T\in \ho_G(\psi ,\varphi)$. Then 
either $T$ is invertible or $T=0$. Consequently, we get \\
(1) If $\psi \nsim \varphi$, then $\ho_G (\psi,\varphi)=\{0\}$; \\
(2) if $\psi =\varphi $, then $T=c 1_V$ for some $c\in \mathbb{C}$. 
\end{lemma}

By Lemma \ref{lem31} and Schur's Lemma \ref{schur}, 
we can get the following corollary. 

\begin{corollary} \label{cor32}
Let $\psi ,\varphi$ be irreducible representations of a finite group $G$.  Then 
\[ \dim \ho_G (\psi, \varphi )=\begin{cases} 
0,& \psi \nsim \varphi ,\\
1,& \psi \sim \varphi . 
\end{cases} \]
\end{corollary} 

\begin{proof}
If $\psi \nsim \varphi$, it follows by Schur's lemma that 
$\ho_G(\psi ,\varphi )=\{0\}$. 
If $\psi \sim \phi$, by Lemma \ref{lem31}, we get 
$\ho_G(\psi ,\varphi ) \cong \ho_G (\varphi ,\varphi)$. 
By Schur's lemma again, we have 
$ \dim \ho_G (\psi ,\varphi) =\dim 
\ho_G (\varphi ,\varphi) =1$. 
 \end{proof}

\begin{lemma} \label{lem3}
Let $A$ be a linear map on vector space $V$. 
Then $\mathrm{rank} (A^2)=\mathrm{rank}(A)$ if and only if  
$ V =\mathrm{Ker }(A) \oplus \mathrm{Im} (A)$. 
In particular,  if $A^2=A$, then $\mathrm{rank}(A)=\mathrm{Tr}(A)$. 
\end{lemma}

\begin{proof}
Denote $\dim ( \mathrm{Im}(A) )=r$ and let 
$A\varepsilon_1,A\varepsilon_2,\ldots ,A\varepsilon_r$ 
be a basis of $\mathrm{Im}(A)$.  
Then $\mathrm{Im}(A^2)=\{ Av:v\in \mathrm{Im}(A)\}=\mathrm{Span}
\{ A^2\varepsilon_1,A^2\varepsilon_2,\ldots ,A^2\varepsilon_r \}$. 

We first show the necessity. Since $\dim (\mathrm{Im}(A^2))=\mathrm{rank}(A^2) 
=\mathrm{rank}(A)=r$, 
then $\mathrm{Im}(A^2)\subseteq \mathrm{Im}(A)$ implies 
$\mathrm{Im}(A^2)=\mathrm{Im}(A)$ and $A^2\varepsilon_1,A^2\varepsilon_2,\ldots ,A^2\varepsilon_r$ 
are linearly independent.   
For every $v\in V$, then $Av\in \mathrm{Im} (A)=\mathrm{Im}(A^2)$, there 
are $k_1,k_2,\ldots ,k_r$ such that 
$Av =\sum_{i=1}^r k_i A^2\varepsilon_i$, 
which implies $A (v-\sum_{i=1}^r k_iA\varepsilon_i)=0$. 
Let $\beta :=v-\sum_{i=1}^r k_iA\varepsilon_i $ 
and $\gamma :=\sum_{i=1}^r k_iA\varepsilon_i $. 
Therefore $v=\beta +\gamma$ and $\beta \in \mathrm{Ker}(A),\gamma 
\in \mathrm{Im}(A)$.  
So $V=\mathrm{Ker}(A) +\mathrm{Im}(A)$. 
Combining the dimensional 
formula $\dim (\mathrm{Ker}(A)) +\dim (\mathrm{Im}(A))=\dim V$, 
it yields that $V=\mathrm{Ker}(A) \oplus \mathrm{Im}(A)$. 

We second show the sufficiency. 
If linear map $A$ 
satisfy $V=\mathrm{Ker}(A) \oplus \mathrm{Im}(A)$,  
we next prove  $\mathrm{rank} (A^2)=\mathrm{rank}(A)$, 
it is sufficient to show that 
$A^2\varepsilon_1,A^2\varepsilon_2,\ldots ,A^2\varepsilon_r$ 
are linearly independent.  If $\sum_{i=1}^r c_i A^2 \varepsilon_i=0$, 
then $\sum_{i=1}^r c_iA\varepsilon_i \in \mathrm{Ker}(A)$.  On the other hand, 
$\sum_{i=1}^r c_iA\varepsilon_i \in \mathrm{Im}(A)$.  
Since $V=\mathrm{Ker}(A) \oplus \mathrm{Im}(A)$, then  
$\mathrm{Ker}(A)\cap \mathrm{Im}(A)=\{0\}$. 
Therefore, $\sum_{i=1}^r c_iA\varepsilon_i =0$. 
As $A\varepsilon_1,A\varepsilon_2,\ldots ,A\varepsilon_r$ 
is a basis of $\mathrm{Im}(A)$, which implies $c_1=c_2=\cdots =c_r=0$.  
So $A^2\varepsilon_1,A^2\varepsilon_2,\ldots ,A^2\varepsilon_r$ 
are linearly independent.   

If $A$ is a linear map such that 
$A^2=A$, we now show $\rank (A)=\tr (A)$. 
Without loss of generality,  
we may assume that $A\in M_{n\times n}(\mathbb C)$ with $\rank (A)=r$.  
By the previous conclusion, we get $\mathbb{C}^n =\mathrm{Ker}(A) 
\oplus \mathrm{Im}(A)$.  Let $\alpha_1,\alpha_2,\ldots ,\alpha_r$ be 
a basis of $\mathrm{Im}(A)$ and $\alpha_{r+1},\ldots ,\alpha_n$ 
be a basis of $\mathrm{Ker}(A)$.  Since $A^2=A$, 
it is not difficult to see that  $A(\alpha_i)=\alpha_i$ for every $i=1,2,\ldots ,r$.  
Then 
\[ A(\alpha_1,\alpha_2,\ldots ,\alpha_n)=(\alpha_1,\alpha_2,\ldots ,\alpha_n)B,\]
where $B=\mathrm{diag}(1,1,\ldots,1,0,\ldots ,0)$.  Therefore, we get  $\rank (A)=\rank (B)=
\tr (B)=\tr (A)$.  This completes the proof. 
\end{proof}

\begin{lemma}  \label{lem4}
Let $\rho : G \to GL(V)$ be a representation of a finite group $G$.  
Define the fixed subspace
$V^G = \{v \in V \,| \, \rho_g v=v, \forall g \in G\}$. 
Then 
\[ \dim V^G=\frac{1}{|G|} \sum\limits_{g\in G} \tr (\rho_g). \]
\end{lemma}

\begin{proof}
It is easy to see that $V^G$ is a $G$-invariant subspace. 
We now denote $P=\frac{1}{|G|} \sum_{g\in G} \rho_g$, 
which is a linear map on $V$. 
 Firstly, $P^2=P$, since 
 \begin{align*}
 P^2&=\left( \frac{1}{|G|} \sum\limits_{h \in G} \rho_h \right)
 \left( \frac{1}{|G|}\sum\limits_{ g\in G} \rho_g \right) \\
 &=\frac{1}{|G|} \sum\limits_{ h\in G} 
 \left(  \frac{1}{|G|} \sum\limits_{g\in G} \rho_h \rho_g \right) \\
 &=\frac{1}{|G|} \sum\limits_{ h\in G} 
 \left(  \frac{1}{|G|} \sum\limits_{g\in G} \rho_{hg} \right) .
 \end{align*}
 We now apply a change of variable by setting $x=hg$. 
 For fixed $h$, as $g$ ranges over all elements of $G$, 
 then $x$ ranges over all elements of $G$. Thus, 
 \[ P^2 =\frac{1}{|G|} \sum\limits_{ h\in G}  
 \left(  \frac{1}{|G|} \sum\limits_{x\in G} \rho_x \right)= 
 \frac{1}{|G| } \sum\limits_{h\in G} P=P. \]
 We next show that $\mathrm{Im}(P)=V^G$. 
 For every $g\in G$ and $v\in V$, observe that 
 \[ \rho_g (Pv)=\rho_g \left(  \frac{1}{|G|} \sum\limits_{h\in G} \rho_h  (v)\right) 
 =\frac{1}{|G|} \sum\limits_{h\in G} \rho_g\rho_h (v) = 
 \frac{1}{|G|} \sum\limits_{x \in G} \rho_x (v)=Pv. 
  \]
  In other words, $Pv\in V^G$ for every $v\in V$, i.e., 
  $\mathrm{Im}(P)=\{Pv : v\in V\} \subseteq V^G $.  
  For each $v\in V^G$, we have $Pv
  =\frac{1}{|G|} \sum_{ g\in G} \rho_g (v) 
  =\frac{1}{|G|} \sum_{g\in G} v=v$. Then 
  $\mathrm{Im}(P) \supseteq \{Pv : v\in V^G\}=\{v:v\in V^G\}=V^G$.  
  Therefore, we get $P^2=P$ and $\mathrm{Im}(P)=V^G$. 
  By Lemma \ref{lem3},  it follows that 
  \[  \dim V^G=\dim (\mathrm{Im}(P)) =\mathrm{rank}(P)=\tr (P)=
  \frac{1}{|G|} \sum\limits_{g\in G} \tr (\rho_g). \]
  Hence, we complete the proof. 
\end{proof}

\begin{corollary} \label{cor28}
Let $\rho : G\to GL(V)$ be a representation and 
$\chi_1$ be the trivial character of $G$. 
Then $\langle \chi_{\rho}, \chi_1 \rangle =\dim V^G$.
\end{corollary} 

\begin{proof}
By the definition of inner product and Lemma \ref{lem4}, 
we have 
\[  \langle \chi_{\rho}, \chi_1 \rangle =
\frac{1}{|G|} \sum_{g\in G} \chi_{\rho}(g) \overline{\chi_1(g)} 
=\frac{1}{|G|} \sum_{g\in G} \tr (\rho_g)=\dim V^G. \]
The desired result holds. 
\end{proof}

\noindent 
{\bf Remark.}~~Corollary \ref{cor28} is a key proposition 
in \cite[p. 86]{BS12},  we here present a versatile and succinct proof. 
Of course, the proof provided by 
Steinberg is quite different and technical. 

\begin{lemma} \label{lem5}
Let $\psi: G \to U_n (\mathbb C)$ and  $\varphi : G \to U_m(\mathbb C) $ be 
unitary matrix representations of a finite group $G$.
Let $V=M_{m\times n} (\mathbb C)$, and define 
a representation $\rho : G \to GL(V)$ by $\rho_g(A)=\varphi_g A \psi^{-1}_{g}$ for every $A\in V$.
Then $V^G=\ho_G ( \psi, \varphi) $ and 
$ \chi_\rho (g)=\chi_\varphi (g) \overline {\chi_\psi(g)}$. 
\end{lemma}

\begin{proof}
First, we can see from the definition of $V^G$ in Lemma \ref{lem4} that 
\begin{align*}
V^G &=\{A\in V \,|\, \rho_g(A)=A,\forall g\in G\} \\
&=\{ A\in M_{m\times n}(\mathbb C) \,|\, \varphi_g A =A \psi_g,\forall g \in G\} \\
&=\ho_G (\psi ,\varphi). 
\end{align*}
Let $ (E_{11},E_{12},\dots,E_{m n})$ be a basis of $ M_{m\times n} (\mathbb C)$. 
Then
 $$\rho_g(E_{ij})=\varphi_g E_{ij} \psi_{g}^{-1} =(\varphi_g e_i)( e_j^\top \psi_{g}^{-1})=
 \sum_{k=1}^{m}\sum_{\ell=1}^{n}\varphi_{ki}(g)(\psi_{g}^{-1})_{j\ell}E_{k\ell}.$$
 Therefore, we can see that the diagonal entries of corresponding matrix of 
$\rho_g$ are $\varphi_{ii}(g)(\psi_g^{-1})_{jj}$ with $i=1,2,\ldots ,m$ and $j=1,2,\ldots ,n$. We get 
\begin{align*}
\chi_\rho(g)
&=\tr(\rho_g)= \sum_{i=1}^{m}\sum_{j=1}^{n} \varphi_{ii}(g)(\psi_g^{-1})_{jj}  \\
&= \sum_{i=1}^{m}\sum_{j=1}^{n} \varphi_{ii}(g)\overline{\psi_{jj}(g)} =  \chi_\varphi (g) \overline {\chi_\psi(g)},
\end{align*}
the third equality holds due to $\psi$ is unitary. 
\end{proof}

\begin{proposition} \label{lem6}
Let $\psi$ and $\varphi $ be representations of a finite group $G$. 
Then $\ho_G(\psi, \varphi) $  is isomorphic to $ \ho_G(\varphi, \psi)$.
\end{proposition}

\begin{proof}
Let $\pi:\ho_G(\psi, \varphi) \longrightarrow \ho_G(\varphi, \psi)$ be a map 
defined by $\pi (T)=\overline{T}^{\top}:=T^*$. 
Since $T \in \ho_G(\psi,\varphi)$, 
that is, $T\psi_g=\varphi_gT$, by taking the conjugate transpose, 
we have $\psi_{g}^*T^*=T^*\varphi_g^*$. 
As $\psi, \varphi$ are unitary representations,  then $\psi_g^*=
\psi_g^{-1}$ and  $\varphi_g^*=\varphi_g^{-1}$, we then get  
$\psi_g^{-1}T^*=T^*\varphi_{g}^{-1},$ 
By left-multiplying $\psi_g$ and right-multiplying $\varphi_g$,  
it follows that $T^*\varphi_g=\psi_gT^*$, 
so we get $T^* \in \ho_G(\varphi, \psi)$. 
It is easy to see that $\pi$ is  an injective and surjective linear map.  
This completes the proof.
\end{proof}

\begin{proposition} \label{lem7}
Let  $\psi$ and $\varphi $ be representations of a finite group $G$.  
Then $\langle \chi_{\varphi}, \chi_{\psi} \rangle$ is a nonnegative integer 
and $\langle \chi_{\psi},\chi_{\varphi} \rangle 
=\langle \chi_{\varphi}, \chi_{\psi} \rangle$.   
\end{proposition}

\begin{proof}
Let $\varphi^{(1)},\varphi^{(2)}, \ldots, \varphi^{(s)} $ be a complete set of 
representatives of the equivalence classes of 
irreducible representations of $G$. 
By Maschke's Theorem,  
we can write $\varphi \sim m_1\varphi^{(1)} \oplus m_2\varphi^{(2)} 
\oplus \cdots \oplus m_s\varphi^{(s)}$ and 
$\psi \sim n_1\psi^{(1)} \oplus n_2\psi^{(2)} 
\oplus \cdots \oplus n_s\psi^{(s)}$ 
for some $m_i, n_i \in \mathbb{N}$.  Then 
$\chi_{\varphi} =\sum_{i=1}^s m_i\chi_i$ and $\chi_{\psi}=
\sum_{i=1}^s n_i\chi_i$. Therefore,  
$$\langle \chi_{\varphi}, \chi_{\psi} \rangle= 
\left\langle  \sum_{i=1}^s m_i\chi_i, 
\sum_{j=1}^s n_j\chi_j \right\rangle = 
\sum\limits_{i=1}^s \sum\limits_{j=1}^s m_i\overline{n_j} 
\langle  \chi_i, \chi_j \rangle =\sum\limits_{i=1}^s m_in_i.$$
Hence, $\langle \chi_{\varphi}, \chi_{\psi} \rangle$ is a nonnegative integer  
and $ \langle \chi_{\varphi}, \chi_{\psi} \rangle = 
\overline{\langle \chi_{\varphi}, \chi_{\psi} \rangle} =
\langle \chi_{\psi},\chi_{\varphi} \rangle $. 
\end{proof}

\begin{theorem} \label{thm1}
Let $\psi $ and $\varphi$ be representations of finite group $G$. Then 
\[  \dim \ho_G (\psi ,\varphi )=\langle \chi_{\psi}, \chi_{\varphi} \rangle. \]
\end{theorem}

\begin{proof}
By Lemma \ref{lem5} and Lemma \ref{lem4}, we have 
\begin{equation} \label{eeq}
\begin{aligned} 
\dim\ho_G(\psi,\varphi) & =\dim V^G= \frac{1}{|G|} \sum\limits_{g\in G} \tr (\rho_g) 
=\frac{1}{G}\sum_{g \in G}\chi_{\rho}(g) \\
&=\frac{1}{G}\sum_{g \in G}\chi_{\varphi}(g)\overline{\chi_\psi(g)} 
=\langle \chi_{\varphi}, \chi_{\psi} \rangle.
\end{aligned}
\end{equation}
By Lemma \ref{lem6} and Lemma \ref{lem7}, we get 
\[ \dim\ho_G(\varphi,\psi)=\dim\ho_G(\psi,\varphi) =
\langle \chi_{\varphi}, \chi_{\psi} \rangle =
\langle  \chi_{\psi}, \chi_{\varphi} \rangle . \]
This required equality holds. 
\end{proof}

\begin{lemma}[Direct sum of representations]
Let $\varphi :G\to GL (W)$ and $\psi : G\to GL (V)$ be representations of  group $G$. 
Define $\varphi \oplus \psi :G\to GL(W\oplus V)$ by 
\[ (\varphi \oplus \psi )_g (w,v)=(\varphi_g(w),\psi_g(v)). \]
Then $\varphi \oplus \psi $ is a representation  and 
$\chi_{\varphi \oplus \psi} (g)= \chi_{\varphi}(g) +\chi_{\psi}(g)$. 
\end{lemma}

\begin{proof}
For any $g_1,g_2\in G$, we have 
\begin{align*}
(\varphi \oplus \psi )_{g_1g_2}(w,v) 
&= \bigl( \varphi_{g_1g_2}(w),\psi_{g_1g_2}(v) \bigr) \\
&= \bigl( \varphi_{g_1}\varphi_{g_2}(w),\psi_{g_1}\psi_{g_2}(v) \bigr) \\
&=(\varphi \oplus \psi)_{g_1} \bigl( \varphi_{g_2}(w), \psi_{g_2}(v)\bigr) \\
&= (\varphi \oplus \psi)_{g_1} (\varphi \oplus \psi)_{g_2}(w,v). 
\end{align*}
Thus, $(\varphi \oplus \psi )_{g_1g_2} 
= (\varphi \oplus \psi)_{g_1} (\varphi \oplus \psi)_{g_2}$. Furthermore, 
\[ \chi_{\varphi \oplus \psi}(g)=\tr (\varphi \oplus \psi)_g=\tr (\varphi_g) + 
\tr (\psi_g) = \chi_{\varphi}(g) +\chi_{\psi}(g).  \]
The desired equality now follows.
\end{proof}

\begin{corollary}
$\ho_G(\varphi \oplus\psi,  \rho) $ is isomorphic to 
$ \ho_G(\varphi, \rho) \oplus \ho_G(\psi, \rho)$. 
\end{corollary}

\begin{proof}
It is sufficient to show that these two vector spaces have same dimension. 
By Theorem \ref{thm1}, we have 
\begin{align*}
\dim (\ho_G(\varphi \oplus \psi,  \rho))&=\langle \chi_{\varphi \oplus \psi}, \chi_{\rho}\rangle 
=\langle \chi_{\varphi}, \chi_{\rho} \rangle
+\langle \chi_{\psi}, \chi_{\rho} \rangle \\
&=\dim \ho_G(\varphi,  \rho)+\dim \ho_G(\psi,  \rho)\\
&=\dim (\ho_G(\varphi, \rho) \oplus \ho_G(\psi, \rho)).
\end{align*}
This complete the proof. 
\end{proof}

By the same line of proof, we can get the following Corollary \ref{cor216}. 

\begin{corollary} \label{cor216}
$\ho_G(\varphi,  \psi \oplus \rho) $ is isomorphic to 
$ \ho_G(\varphi, \psi) \oplus \ho_G(\varphi, \rho).$
\end{corollary}

In Lemma \ref{lem5}, 
we get $\chi_\rho (g)=\overline {\chi_\psi(g)} \chi_\varphi (g) $. 
We next show a more generalized result [Theorem \ref{thm216}], 
which states that $\rho$ is equivalent to 
the tensor representation $\overline{\psi}\otimes \varphi$. 
For convenience, we now introduce the conjugate representation and 
tensor representation. By Lemma \ref{lem21}, without loss of generality, 
we may only consider  the  (unitary) matrix representations sometimes.

\begin{lemma}[Conjugate representation] \label{lem214}
Let $\psi :G \to GL_n (\mathbb{C})$ be a representation of  group $G$. 
Define $\overline{\psi }: G \to GL_n(\mathbb{C})$ by 
$\overline{\psi}_g =\overline{\psi_g}$.  
Then $\overline{\psi}$ is a representation and 
$\chi_{\overline{\psi}}(g)=\overline{\chi_{\psi}(g)}$. 
\end{lemma}

\begin{proof}
We first can see that $\overline {\psi} $ is a homomorphism, 
since for every $g_1, g_2 \in G$, \\
then   $\overline {\psi}_{(g_1g_2)}=\overline{\psi_{(g_1g_2)}}
=\overline{\psi_{g_1} }\overline{\psi_{g_2}} 
=\overline{\psi_{g_1}} \cdot \overline{\psi_{g_2}} 
=\overline{\psi}_{g_1}\cdot\overline{\psi}_{g_2}$, 
where the second equality follows by noting that $\psi$ is a representation.  
Hence $\overline{\psi}$ is a representation.  
Furthermore, we have 
$\chi_{\overline{\psi}}(g)=\tr(\overline{\psi}_g)
=\tr(\overline{\psi_g})=\overline{\tr(\psi_g)}=\overline{\chi_{\psi}(g)}$. 
\end{proof}

\begin{lemma}[Tensor representation]  \label{lem215}
Let $\varphi: G \to GL_m(\mathbb C)$ and $\psi: G \to GL_n(\mathbb C)$
be representations of  $G$. 
Define $\varphi \otimes \psi: G \to GL( \mathbb C^m \otimes \mathbb C^n)$ by 
\[ (\varphi \otimes \psi)_g (x \otimes y)=(\varphi_g(x)) \otimes (\psi_g (y)).\] 
Then $\varphi \otimes \psi$ is a representation  and 
$\chi_{\varphi \otimes \psi}(g) =\chi_{\varphi}(g)\chi_{\psi}(g)$. 
\end{lemma}

\begin{proof}
For any $g_1,g_2 \in G$, we have 
\begin{align*}
(\varphi \otimes \psi)_{g_1g_2}(x \otimes y) &=(\varphi_{g_1g_2}(x)) \otimes( \psi_{g_1g_2}(y))\\
&=(\varphi_{g_1}\varphi_{g_2}(x) )\otimes (\psi_{g_1}\psi_{g_2}(y))\\
&=(\varphi_{g_1} \otimes \psi_{g_1})\bigl( (\varphi_{g_2}(x))\otimes(\psi_{g_2}(y)) \bigr)\\
&=(\varphi \otimes \psi)_{g_1}(\varphi \otimes \psi)_{g_2}(x\otimes y).
\end{align*}
Therefore, $ (\varphi \otimes \psi)_{g_1g_2}=
(\varphi \otimes \psi)_{g_1}(\varphi \otimes \psi)_{g_2}$, 
hence $\varphi \otimes \psi$ is a representation. 
Moreover, we have 
\[  \chi_{\varphi \otimes \psi} (g) =\tr(\varphi_g \otimes  \psi_g)=\tr(\varphi_g)\tr(\psi_g)
=\chi_{\varphi} (g)\chi_{\psi}(g). \]
Sometimes, we may write 
$\chi_{\varphi \otimes \psi}=\chi_{\psi}\chi_{\varphi}=\chi_{\varphi}\chi_{\psi}$ in symbol.
\end{proof} 

\begin{theorem} \label{thm216}
Let $\varphi : G \to U_m(\mathbb C)$ and $\psi : G \to U_n( \mathbb C)$ 
be unitary matrix representations of a finite group $G$ 
and let $V=M_{m \times n}(\mathbb C)$. 
If  $\rho : G \to GL(V)$ is defined by $\rho_g (A)
=\varphi_g A \psi^{-1}_{g}$ for every $A\in V$, 
then $\rho \sim \overline{\psi} \otimes \varphi $. 
Consequently, we get $\chi_{\rho}(g)=\overline{\chi_{\psi}(g)} \chi_{\varphi}(g)$. 
\end{theorem}

\begin{proof}
Define a map $T:\mathbb C^n \otimes \mathbb{C}^m \to M_{m \times n}(\mathbb C)$ by 
setting 
$T(v \otimes w)$$=(v_1w, v_2w, \dots, v_nw)$ for 
all $v=(v_i)\in \mathbb{C}^n$ and $w \in \mathbb{C}^m$, i.e., 
$T(v \otimes w) \cdot e_j=v_jw=(v^{\top} e_j)w$,
in fact, for any $x \in \mathbb C^n$, we have $T(v \otimes w) x=(v^\top x)w$.

Our goal is to show that $\rho_g T=T(\overline \psi \otimes \varphi)_g$ for every $g\in G$. 
In other words,  for any $v\in \mathbb{C}^n$ and $w\in \mathbb{C}^m$,  it is sufficient to prove 
\[ \rho_g T(v \otimes w) =T(\overline \psi \otimes \varphi)_g (v \otimes w). \]
For each $e_j \in \mathbb{C}^n$, we have 
\begin{align*}
T(\overline \psi \otimes \varphi)_g (v \otimes w) {e_j}
&=T(\overline\psi_g(v) \otimes \varphi_g(w))e_j 
=\bigl( (\overline \psi_g(v))^\top e_j\bigr)\varphi_g(w)\\
&=(v^\top (\overline{\psi_g})^\top e_j) \varphi _g(w)=(v^\top \psi_g^{-1}e_j)\varphi_g(w)\\
&=\varphi_g (v^\top \psi_g^{-1}e_j) w=\varphi_g T(v \otimes w)(\psi_g^{-1} e_j)\\
&=\rho_gT(v \otimes w) e_j. 
\end{align*}
Hence, the required equality immediately holds. 

Moreover, by Lemma \ref{lem214} and Lemma \ref{lem215}, we obtain 
\[ \chi_{\rho}(g)=\chi_{\overline{\psi} \otimes \varphi}(g) 
=\chi_{\overline{\psi} }(g)\chi_{\varphi }(g)=\overline{\chi_{\psi}(g)}\chi_{\varphi}(g). \] 
This complete the proof. 
\end{proof}

\section{Proofs of orthogonality relations}

\begin{proposition} \label{prop31}
Let $\varphi: G \to U_m(\mathbb C)$ and $\psi: G \to U_n(\mathbb C)$
be unitary representations of  $G$. 
Define $P:\ho (\mathbb C^m, \mathbb C^n)\to\ho (\mathbb C^m, \mathbb C^n)$, by setting
\[ P(T)=\frac{1}{|G|}\sum_{g\in G}\psi_g T \varphi_g^{-1}.\] 
Then $P$ is a projection with respect to the subspace $\ho_G(\varphi, \psi)$. 
\end{proposition}

\begin{proof}
We first verify that  
$P(T) \in \ho_G(\varphi, \psi)$ for every $T \in \ho (\mathbb C^m, \mathbb C^n)$. 
Notice that $P(T)=\frac{1}{|G|}\sum_{g\in G}\psi_g T \varphi_g^{-1}
=\frac{1}{|G|}\sum_{g\in G}\psi_g^{-1} T \varphi_g$,
then by a direct computation, we obtain 
\begin{align*}
P(T)\varphi_h &=\frac{1}{|G|}\sum_{g\in G}\psi_g^{-1} T \varphi_g\varphi_h 
=\frac{1}{|G|}\sum_{g\in G}  \psi_g^{-1} T \varphi_{gh}\\
&=\frac{1}{|G|}\sum_{x\in G}  \psi_{xh^{-1}}^{-1} T \varphi_{x} 
=\frac{1}{|G|}\sum_{x\in G}  \psi_{hx^{-1}} T \varphi_{x} \\
&=\frac{1}{|G|}\sum_{g\in G}\psi_{hg}T \varphi_{g}^{-1} 
=\psi_h\frac{1}{|G|}\sum_{g\in G}\psi_{g}T \varphi_{g}^{-1}
=\psi_h P(T)
\end{align*}
for all $g\in G$. This proves $P(T) \in \ho_G(\varphi, \psi)$.

We next  prove that $P(T)=T$ for every $T \in \ho_G(\varphi, \psi) $ just by noting 
\[ P(T)=\frac{1}{|G|}\sum_{g\in G}\psi_g T \varphi_g^{-1}
=\frac{1}{|G|}\sum_{g\in G} T\varphi_g\varphi_g^{-1}=T.\]
Finally, we show $P$ is idempotent, i.e.,  $P^2=P$. 
For any $T\in \ho (\mathbb{C}^m, \mathbb{C}^n)$, 
it is easy to see that
$P^2(T)=P(P(T))=P(T)$, since $P(T)\in \ho_{G}(\varphi, \psi)$ 
and $P(S)=S$ for all $S\in \ho_G (\varphi ,\psi)$. 
This complete the proof. 
\end{proof}

{\bf Remark.}~~
By Proposition \ref{prop31} and Lemma \ref{lem3}, we get 
\[  \dim \ho_G (\varphi ,\psi )=\dim \mathrm {Im}P=\mathrm{rank} (P)=\tr (P).\]
Let $E_{11}, E_{12},\dots, E_{nm}$ be a basis of $\ho(\mathbb C^m, \mathbb C^n)$. 
Then 
\begin{align*} 
P(E_{ij}) &=\frac{1}{|G|} \sum\limits_{g\in G} \psi_g E_{ij} \varphi_g^{-1} 
= \frac{1}{|G|} \sum\limits_{g\in G} (\psi_g e_i)(e_j^{\top} \varphi_g^{-1})  \\
&= \frac{1}{|G|} \sum\limits_{g\in G} \sum\limits_{k=1}^n \sum\limits_{\ell =1}^m 
\psi_{ki}(g) (\varphi_g^{-1})_{j\ell} E_{k\ell}. 
\end{align*}
Therefore, we obtain 
\begin{align*}
\tr(P)&=\sum_{i=1}^{n} \sum_{j=1}^{m} 
\left( \frac{1}{|G|} \sum _{g\in G}\psi_{ii}(g)\varphi_{jj}(g^{-1}) \right)
=\frac {1}{|G|}\sum_{g \in G}\tr(\psi_g)\tr (\varphi_g^{-1})\\
&=\frac {1}{|G|}\sum_{g \in G}\tr(\psi_g)\tr (\overline{\varphi_g}^\top)
=\frac {1}{|G|}\sum_{g \in G}\chi_{\psi}(g) \overline{\chi_{\varphi}(g)} 
=\langle \chi_{\psi}, \chi_{\varphi} \rangle.
\end{align*}
Thus, we get $\dim \ho_G (\varphi ,\psi )=\langle \chi_{\psi}, \chi_{\varphi} \rangle
=\langle \chi_{\varphi}, \chi_{\psi} \rangle$.

\begin{corollary}[First orthogonality relation]
Let $\varphi $ and $\psi $ be two irreducible representations of a finite group G. Then
\[ \langle \chi_{\varphi}, \chi_{\psi} \rangle  =
 \begin{cases}                                    
1, & {\varphi \sim \psi,} \\
0, &  {\varphi \nsim \psi.}
\end{cases} 
\]
\end{corollary}

\begin{proof}
Combining Theorem \ref{thm1} and Corollary \ref{cor32}, we get 
\[ \langle \chi_{\varphi}, \chi_{\rho}\rangle=\dim \ho_G(\varphi, \psi) =
\begin{cases}
1, & {\varphi \sim \psi,} \\
0, &  {\varphi \nsim \psi.}
\end{cases}
\]
Thus, the first orthogonality relation immediately follows.
\end{proof}

\begin{corollary}[Schur's orthogonality relation]
Let $\varphi : G \to U_m(\mathbb C)$ and $\psi :G \to U_n(\mathbb C)$ 
be inequivalent irreducible unitary  representations. Then \\
(1) $\langle \psi_{ki}, \varphi_{\ell j} \rangle=0$ for all $i,j,k,\ell$;\\
(2) $\langle \psi_{ki}, \psi_{\ell j} \rangle = 
\begin{cases}
1/n, & \text{ if $k=\ell $ and $i=j$}, \\
0 ,& \text{otherwise}. 
\end{cases} $
\end{corollary}

\begin{proof}
By  Proposition \ref{prop31} and Schur's Lemma \ref{schur}, then for any $i,j$, we have 
 $P(E_{ij})\in \ho_G (\varphi ,\psi)=\{0\}$. 
By the remark of Proposition \ref{prop31}, 
then for any $k,\ell$, we have 
\[ \langle \psi_{ki}, \varphi_{\ell j} \rangle 
= \frac{1}{|G|} \sum\limits_{g\in G} \psi_{ki}(g) \overline{\varphi_{\ell j}(g)} =
\frac{1}{|G|} \sum\limits_{g\in G} \psi_{ki}(g) (\varphi_g^{-1})_{j\ell }=0. \] 
For (2), by  setting $\varphi =\psi$ in Proposition \ref{prop31}, we get 
\[ P(E_{ij}) =  \frac{1}{|G|}\sum_{g\in G}\psi_g E_{ij} \psi_g^{-1} \in \ho_G (\psi ,\psi).  \]
By Schur's Lemma \ref{schur}, then $P(E_{ij})=cI_n$ for some constant $c\in \mathbb{C}$. 
By taking trace, we have $c=\frac{1}{n}\tr P(E_{ij})=\frac{1}{n} \tr (E_{ij})$. Then 
\[ P(E_{ij}) = \sum\limits_{k=1}^n \sum\limits_{\ell =1}^m 
 \left(  \frac{1}{|G|} \sum\limits_{g\in G} 
\psi_{ki}(g) \overline{\psi_{\ell j}(g)}\right) E_{k\ell} =\frac{1}{n}\tr (E_{ij})I_n. \]
For $k\neq \ell$, we have $\langle \psi_{ki}, \psi_{\ell j} \rangle =0$; 
For $k=\ell$, we have 
\[   \frac{1}{|G|} \sum\limits_{g\in G}   
\psi_{ki}(g) \overline{\psi_{k j}(g)} =\frac{1}{n} \tr (E_{ij}).  \]
Then  $\langle \psi_{ki}, \psi_{k j} \rangle =\frac{1}{n}$ for $i=j$ and 
$\langle \psi_{ki}, \psi_{\ell j} \rangle =0$ for $i\neq j$. 
\end{proof}

\section*{Acknowledgments}
All authors are grateful for valuable comments from the referee, 
which considerably improve the presentation of our manuscript.
This work was supported by  NSFC (Grant No. 11671402, 11871479),  
Hunan Provincial Natural Science Foundation (Grant No. 2016JJ2138, 2018JJ2479) 
and  Mathematics and Interdisciplinary Sciences Project of CSU.

\end{document}